\documentclass[a4paper,12pt,reqno]{amsart}
\usepackage{latexsym,amscd,amssymb,url}
\usepackage{xcolor}
\definecolor{dblue}{rgb}{0,0,.6}
\usepackage[colorlinks=true, linkcolor=dblue, citecolor=dblue, filecolor = dblue, menucolor = dblue, urlcolor = dblue]{hyperref}

\usepackage[all,cmtip]{xy}  
\usepackage{graphicx}
\usepackage{color}
\usepackage{hyperref}
\usepackage{enumerate}
\usepackage{appendix}

\newtheorem{theorem}{Theorem}[section]

\theoremstyle{plain}

\newtheorem{conjecture}[theorem]{Conjecture}
\newtheorem{corollary}[theorem]{Corollary}

\newtheorem{example}[theorem]{Example}

\newtheorem{lemma}[theorem]{Lemma}

\newtheorem{remark}[theorem]{Remark}

\newcommand{\del}{\partial}
\newcommand{\Z}{\mathbb Z}

\newcommand{\C}{\mathbb C}

\newcommand{\R}{\mathbb R}

\newcommand{\Hom}{\operatorname{Hom}}
\newcommand{\Pic}{\operatorname{Pic}}

\newcommand{\Aut}{\operatorname{Aut}}

\newcommand{\Alb}{\operatorname{Alb}}

\newcommand{\codim}{\operatorname{codim}}

\newcommand{\supp}{\operatorname{supp}}
 
\newcommand{\red}{\operatorname{red}}

\newcommand{\Char}{\operatorname{Char}}

\newcommand{\Loc}{\operatorname{Loc}}

\begin{document}   

\title[Holomorphic one-forms and topology]{Zeros of holomorphic one-forms and topology of K\"ahler manifolds} 

\author[Stefan Schreieder]{Stefan Schreieder\\
Appendix written jointly with Hsueh-Yung Lin
}

\address{Mathematisches Institut, LMU M\"unchen, Theresienstr. 39, 80333 M\"unchen, Germany}
\email{schreieder@math.lmu.de}

\address{Mathematisches Institut, Universit\"at Bonn, Endenicher Allee 60, 53115 Bonn, Germany}
\email{linhsueh@math.uni-bonn.de}

\date{November 4, 2019}
 


\begin{abstract} 
A conjecture of Kotschick predicts that a compact K\"ahler manifold $X$ fibres smoothly over the circle if and only if it admits a holomorphic one-form without zeros.
In this paper we develop an approach to this conjecture and verify it in dimension two. 
In a joint paper with Hao \cite{HS}, we use our approach to prove Kotschick's conjecture for smooth projective threefolds.  
\end{abstract}

\maketitle

\section{Introduction}

This paper is motivated by the following conjecture of Kotschick \cite{Ko}.

\begin{conjecture} \label{conj1}
For a compact K\"ahler manifold $X$, the following are equivalent.
\begin{enumerate}[(A)]
\item $X$ admits a holomorphic one-form without zeros; \label{item:conj1:hol-1-form}
\item \label{item:conj1:circle}
$X$ admits a real closed $1$-form without zeros;
or, by Tischler's theorem \cite{Ti} equivalently, the underlying differentiable manifold is a  $C^\infty$-fibre bundle over the circle.
\end{enumerate}
\end{conjecture}

The implication (\ref{item:conj1:hol-1-form}) $\Rightarrow$ (\ref{item:conj1:circle}) is clear; 
the possibility of the converse implication  (\ref{item:conj1:circle}) $\Rightarrow$ (\ref{item:conj1:hol-1-form}) is asked in \cite{Ko}. 
Condition (\ref{item:conj1:circle}) is equivalent to asking that the smooth manifold that underlies $X$ is a quotient $M\times [0,1]/\sim$, where $M$ is a closed real manifold of odd dimension and $M\times 0$ is identified with $M\times 1$ via some diffeomorphism of $M$.
Kotschick's conjecture relates this purely topological condition  with the complex geometric condition 
 that $X$ has a holomorphic one-form without zeros.

The purpose of this paper is to related Kotschick's conjecture to the following condition
\begin{enumerate}[(A)] 
\setcounter{enumi}{2}
\item there is a holomorphic one-form  $\omega \in H^0(X,\Omega_X^1)$, such that for any finite \'etale cover $\tau:X'\to X$, the sequence 
$$  H^{i-1}(X',\C)\xrightarrow{\wedge \omega'}  H^{i}(X',\C)\xrightarrow{\wedge \omega'}  H^{i+1}(X',\C) ,
$$
given by cup product with $\omega':=\tau^\ast \omega$,
is exact for all $i$.  
\label{item:conj1:exact}
\end{enumerate} 
This is motivated by a theorem of Green and Lazarsfeld \cite[Proposition 3.4]{GL}, who proved the implication (\ref{item:conj1:hol-1-form}) $\Rightarrow$ (\ref{item:conj1:exact}).
Our first result is the following, which in view of Green and Lazarsfeld's theorem yields some positive evidence for Conjecture \ref{conj1}.

\begin{theorem} \label{thm:B=>C}
For any compact K\"ahler manifold $X$, we have (\ref{item:conj1:circle}) $\Rightarrow$ (\ref{item:conj1:exact}).
\end{theorem}

By the above theorem, in order to prove Kotschick's conjecture, it would be enough to 
show that (\ref{item:conj1:exact}) implies (\ref{item:conj1:hol-1-form}).
Compared to the original implication (\ref{item:conj1:circle}) $\Rightarrow$ (\ref{item:conj1:hol-1-form}), this has the major advantage that (\ref{item:conj1:exact}) and (\ref{item:conj1:hol-1-form}) are complex geometric conditions, while (\ref{item:conj1:circle}) is not.
More precisely, it is natural to wonder whether a one-form $\omega\in H^0(X,\Omega_X^1)$ which satisfies condition (\ref{item:conj1:exact}) must be without zeros.
This would have the remarkable implication that the question whether $\omega$ has zeros depends only on the de Rham class of $\omega$ and the homotopy type of $X$.
We show that this is true for surfaces.

\begin{theorem} \label{thm:surfaces}
Let $X$ be a compact K\"ahler surface.
If $\omega\in H^0(X,\Omega_X^1)$ satisfies condition (\ref{item:conj1:exact}), then it has no zeros.
In particular, Conjecture \ref{conj1} holds for compact K\"ahler surfaces.
\end{theorem}

The proof of Theorem \ref{thm:surfaces} uses classification of surfaces.
In the Appendix to this paper, written jointly with Lin, we  give however a more general and direct argument which does not rely on classification results, see Theorem \ref{thm:appendix} below.

In joint work with Hao \cite{HS}, we use the approach developed here to prove Conjecture \ref{conj1} for smooth projective threefolds. 

The following theorem proves some partial results in arbitrary dimension. 

\begin{theorem} \label{thm:C=>...}
Let $X$ be a compact connected K\"ahler manifold with a holomorphic one-form $\omega$ such that the complex $(H^\ast(X,\C),\wedge \omega)$ given by cup product with $\omega$ is exact.
Then the analytic space $Z(\omega)$ given by the zeros of $\omega\in H^0(X,\Omega_X^1)$ has the following properties.
\begin{enumerate}[(1)]
\item For any connected component $Z\subset Z(\omega)$ with $d=\dim Z$,
$$
H^{d}(Z,{\omega_X} |_Z)=0 . 
$$
In particular, $\omega$ does not have any isolated zero. \label{item:thm:Z(omega)}

\item If $f:X\to A$ is a holomorphic map to a complex torus $A$ such that $\omega\in f^\ast H^0(A,\Omega_A^1)$, then $f(X)\subset A$ is fibred by tori. 
\label{item:thm:fibred} 
\end{enumerate}
\end{theorem}

Ein and Lazarsfeld \cite[Theorem 3]{EL} showed that the image of a morphism $f:X\to A$ to a complex torus $A$ is fibred by tori if $\chi(X,\omega_X)=0$ and $\dim f(X)=\dim X$.
In item (\ref{item:thm:fibred}) above we obtain the same conclusion without any assumption on $f$, but where we replace $\chi(X,\omega_X)=0$ by the stronger condition on the exactness of $(H^\ast(X,\C),\wedge \omega)$.

Theorem \ref{thm:B=>C} and item (\ref{item:thm:fibred}) in the above theorem imply for instance that a K\"ahler manifold $X$ with simple Albanese torus $\Alb(X)$ and with $b_1(X)>2\dim(X)$ does not admit a $C^\infty$-fibration over the circle.
Similarly, we obtain the following corollary in the projective case.

\begin{corollary}
Let $X$ be a smooth complex projective variety such the manifold which underlies $X$ fibres smoothly over the circle.
Then there is a surjective holomorphic morphism $f:X\to A$ to a positive-dimensional abelian variety $A$. 
\end{corollary}

The following example of Debarre, Jiang and Lahoz shows that the \'etale covers in condition (\ref{item:conj1:exact}) are necessary to make Theorem \ref{thm:surfaces} true. 

\begin{example}[{\cite[Example 1.11]{DJL}}]
Let $C_1,C_2$ be smooth projective curves with $g(C_1)>1$ and $g(C_2)=1$ and automorphisms $\varphi_i\in \Aut(C_i)$ of order two such that $C_i/\varphi_i$ has genus one for $i=1,2$.
Then the quotient 
$$
X:=(C_1\times C_2)/(\varphi_1\times \varphi_2)
$$ 
has the same rational cohomology ring as an abelian surface, and so $\wedge \omega$ is exact on cohomology for any non-zero $\omega\in H^0(X,\Omega_X^1)$.
However, if $\omega$ is obtained as pullback via the map $\pi:X\to C_1/\varphi_1$, then it vanishes along the multiple fibres of $\pi$, which lie above the branch points of $C_1\to C_1/\varphi_1$. 
\end{example}

\begin{remark}
This paper raises the question whether condition (\ref{item:conj1:exact}) implies (\ref{item:conj1:hol-1-form}). 
In view of \cite[Proposition 3.4]{GL} it is natural to wonder whether more generally, a holomorphic one-form $\omega \in H^0(X,\Omega_X^1)$ such that for any finite \'etale cover $\tau:X'\to X$
$$
H^{i-1}(X',\C)\xrightarrow{\wedge \tau^\ast \omega} H^i(X',\C) \xrightarrow{\wedge \tau^\ast \omega} H^{i+1}(X',\C)
$$
is exact for all $i<c$ implies that $\codim_X(Z(\omega))\geq c$.
This goes back to \cite{BWY}, where it is asked whether equality always holds in  \cite[Theorem 1.1]{BWY}.
However, blowing-up a point in $Z(\omega)$ easily produces counterexamples to this conjecture. 
The question remains open under suitable minimality assumptions on $X$ (e.g.\ if $K_X$ is nef). 
\end{remark}

\subsection*{Why the K\"ahler assumption?}
The K\"ahler assumption in Conjecture \ref{conj1} is essential.
For instance, a Hopf surface $X$ is a compact complex surface with $H^0(X,\Omega_X^1)=0$, whose underlying differentiable manifold is diffeomorphic to $S^1\times S^3$, and so it satisfies (\ref{item:conj1:circle}) but not (\ref{item:conj1:hol-1-form}).

\subsection*{Acknowledgement}
I am grateful to Dieter Kotschick for sending me the preprint \cite{Ko} in spring 2013, where he poses the problem about the equivalence of (\ref{item:conj1:hol-1-form}) and (\ref{item:conj1:circle})  in Conjecture \ref{conj1}.
I am also grateful to Rui Coelho, Feng Hao, Dieter Kotschick and Anand Sawant for useful conversations and to Hsueh-Yung Lin, Mihnea Popa, Christian Schnell, Botong Wang and the referees for useful comments.
This work is supported by the DFG Grant ``Topologische Eigenschaften von Algebraischen Variet\"aten'' (project no.\ 416054549).

\subsection*{Notation}
For a holomorphic one-form $\omega$ on a K\"ahler manifold $X$, we denote by $Z(\omega)$ the (possibly non-reduced) analytic space given by the zeros of $\omega$, viewed as a section of the vector bundle $\Omega_X^1$.

\section{Proof of Theorem \texorpdfstring{\ref{thm:B=>C}}{1.2}} 
\label{sec:thm:B=>C}

Let $X$ be a smooth connected manifold.
We denote by  $\Loc(X)$ the group of local systems on $X$ whose stalks are one-dimensional $\C$-vector spaces. 
Since  local systems on the interval are trivial, the choice of a base point $s\in S^1$ induces a canonical isomorphism $\Loc(S^1)\cong \C^\ast$.
Hence, if we fix a base point $x\in X$, then for any $L\in \Loc(X)$, any continuous map $\gamma:S^1\to X$ with $\gamma(s)=x$ yields a canonical element $\gamma^\ast L\in \Loc(S^1)\cong \C^\ast$, which, as one checks, depends only on the homotopy class of $\gamma$.
%
This construction gives rise to the so called monodromy representation, which (since $X$ is connected)
 induces an isomorphism between $\Loc(X)$ and the character variety
$$
\Char(X):=\Hom(\pi_1(X,x),\C^\ast)\cong H^1(X,\C^\ast).
$$ 

If $L\in \Loc(X)$, then the associated complex line bundle has locally constant transition functions, hence it admits a flat connection and so the first Chern class $c_1(L)$ must be torsion. 
The long exact sequence associated to the short exact sequence $0\to \Z \to \C \to \C^\ast \to 0$ of locally constant sheaves on $X$ thus shows that $\Loc(X)$ is isomorphic to an extension of a finite group by the connected subgroup $\Loc^0(X)\subset \Loc(X)$ which contains the trivial local system.
Moreover,
$$ 
\Loc^0(X)\cong \frac{H^1(X,\C)}{H^1(X,\Z)}\cong (\C^\ast)^{b_1(X)} .
$$ 
coincides with the subgroup $\{L\in \Loc(X) \mid c_1(L)=0\}$. 


\subsection{Local systems associated to closed $1$-forms and Novikov's inequality} \label{subsec:L(alpha)}
If $\alpha$ is a closed complex valued $1$-form on $X$, then we can construct a local system $L(\alpha)\in \Loc^0(X)$ as follows.
Consider the twisted de Rham complex $(\mathcal A^\ast_{X,\C},d+\wedge \alpha)$, where $\mathcal A^k_{X,\C}$ denotes the sheaf of complex valued $C^\infty$-differential $k$-forms on $X$, and where $\wedge \alpha$ acts on a $k$-form $\beta$ via $\beta\mapsto \alpha\wedge \beta$.
There is an open covering $\mathcal U=\{U_i\}_{i\in I}$ of $X$ such that $\alpha|_{U_i}=dg_i$ for some smooth function $g_i$ on $U_i$.
For a $k$-form $\beta$ on $U_i$, we then have $(d+\wedge \alpha)(\beta)=0$ if and only if $d(e^{g_i}\beta)=0$.
This shows that the twisted de Rham complex $(\mathcal A^\ast_{X,\C},d+\wedge \alpha)$ is exact in positive degrees and it resolves a sheaf $L(\alpha)$ whose sections above $U_i$ are given by all smooth functions $f$ with $d (e^{g_i}f)=0$, i.e.\ $f=e^{-g_i}c$ for some constant $c\in \C$.
Hence, $L(\alpha)\in \Loc(X)$ is a local system with stalk $\C$ on $X$.
Moreover, $c_1(L(\alpha))=0$ because the cocycle $(g_i-g_j)\in \check{C}^1(\mathcal U, ( \mathcal A_X^0)^\times)$ maps to zero in $H^2(X,\Z)$ and so
\begin{align} \label{eq:L(alpha)}
L(\alpha)\in \Loc^0(X),
\end{align} 
as we want.

Since $L(\alpha)$ is resolved by the $\Gamma$-acyclic complex $(\mathcal A^\ast_{X,\C},d+\wedge \alpha)$, we find that
\begin{align}\label{eq:H^k(L(alpha))}
H^k(X,L(\alpha))=H^k((A^\ast(X,\C),d+\wedge \alpha)),
\end{align}
where $A^k(X,\C)=\Gamma(X,\mathcal A^k_{X,\C})$.
In view of (\ref{eq:H^k(L(alpha))}), we can define the Novikov Betti numbers $b_i(\alpha)$ of $\alpha$ as follows, cf.\ \cite{Pa} or \cite{Fa}: 
$$
b_k(\alpha):=\dim_\C H^k(X,L(\alpha)) .
$$ 

A closed $1$-form $\alpha$ on $X$ is Morse if locally at each zero $x\in Z(\alpha)$ of $\alpha$, $\alpha=dh$ for some Morse function $h$.
If $\alpha$ is Morse, its Morse index at a zero $x$ is defined as the Morse index of $h$ and we denote by $m_i(\alpha)$ the number of zeros of $\alpha$ of Morse index $i$.
The Novikov inequalities then state the following, see \cite[Theorem 1]{Pa}:

\begin{theorem}[Novikov's inequalities] \label{thm:Novikov-ineq}
Let $X$ be a closed manifold and let $\alpha$ be a closed $1$-form on $X$.
Suppose that $\alpha$ is Morse in the above sense.
Then for sufficiently large $t\in \R$, $m_i( \alpha)\geq b_i(t\alpha)$.
In particular, if $\alpha$ has no zeros, then for $t \gg 0$,
$$
H^i(X,L(t\alpha))=0 \ \ \text{for all $i$.}
$$ 
\end{theorem}

%
%

\subsection{Local systems associated to holomorphic $1$-forms} \label{subsec:L(omega)}
Let now $X$ be a compact K\"ahler manifold.
For any holomorphic $1$-form $\omega$ on $X$,
$\omega$ is closed and so we get a local system $L(\omega)$ as above.
This induces a short exact sequence
\begin{align} \label{eq:ses:Loc^0}
0\longrightarrow H^0(X,\Omega_X^1)\longrightarrow \Loc^0(X)\longrightarrow \Pic^0(X)\longrightarrow 0,
\end{align}
where  $\Loc^0(X) \to \Pic^0(X)$ is given by $L\mapsto L\otimes_\C \mathcal O_X$.

\begin{lemma} \label{lem:coho-Lomega}
Let $X$ be a compact K\"ahler manifold and let $\omega\in H^0(X,\Omega_X^1)$ be a holomorphic $1$-form.
Let $c\in \Z \cup \{\infty \}$ be maximal such that  
$$
H^{i-1}(X,\C)\stackrel{\wedge \omega}\longrightarrow  H^{i}(X,\C)\stackrel{\wedge \omega}\longrightarrow  H^{i+1}(X,\C) 
$$
is exact for all $i < c$.
Then the local system $L(\omega)$ associated to $\omega$ 
satisfies  $H^i(X,L(\omega))=0$ for all $i< c$.
Moreover, if $c\neq \infty$, then $H^{c}(X,L(\omega))\neq 0$.
\end{lemma}

\begin{proof}
The local system $L(\omega)$ is resolved by the following complex
$$
(\Omega_X^\ast,\del+\wedge \omega):=0\longrightarrow \Omega_X^0\xrightarrow{\del+\wedge \omega}   \Omega_X^1\xrightarrow{\del +\wedge \omega}  \dots \Omega_X^n{-1}\xrightarrow{\del+ \wedge \omega }  \Omega_X^n\to 0 .
$$
To see that this complex is exact in positive degrees, one uses that locally $\omega=dh$ and so for any local holomorphic form $\beta$, we have $de^h\beta=e^h(d\beta+dh\wedge \beta)$ and so $\del \beta+\omega\wedge \beta=0$ if and only if $d e^h \beta=0$ and we can use the holomorphic Poincar\'e lemma to prove the claim.
Hence,
$$
H^i(X,L(\omega))=\mathbb H^i(X,(\Omega_X^\ast,\del+\wedge \omega) ) .
$$
There is a spectral sequence 
$$
'E_1^{p,q}:=H^p(X,\Omega_X^q) \Rightarrow \mathbb H^{p+q}(X,(\Omega_X^\ast,\del+\wedge \omega) ) .
$$
The differential $d_1:'E_1^{p,q}\to 'E_1^{p,q+1}$ is induced by $\del +\wedge \omega$.
Since $\del$ acts trivially on $'E_1^{p,q}:=H^p(X,\Omega_X^q)$, we find that $d_1=\wedge \omega$.
It thus follows from \cite[Proposition 3.7]{GL} that the above spectral sequence degenerates at the second page, i.e.\ $'E_2={'E}_\infty$.

Our assumption implies $'E_2^{p,q}=0$ for $p+q< c$ and so $H^i(X,L(\omega))=0$ for $i< c$.
Let us now assume $c\neq \infty$.
By the definition of $c$,
$$
H^{c-1}(X,\C)\stackrel{\wedge \omega}\longrightarrow H^{c}(X,\C)\stackrel{\wedge \omega}\longrightarrow H^{c+1}(X,\C) 
$$
is not exact.
Since $\omega\in H^{1,0}(X)$ is of type $(1,0)$, the above complex respects the Hodge decomposition and so we find that there must be some $j$ such that
$$
H^{j-1,c-j}(X)\stackrel{\wedge \omega}\longrightarrow H^{j,c-j}(X)\stackrel{\wedge \omega}\longrightarrow H^{j+1,c-j}(X) 
$$
is not exact.
Hence $'E_2^{j,c-j}\neq 0$.
Since $'E_2^{j,c-j}={'E}_\infty^{j,c-j}$, we get $H^{c} (X,L(\omega))\neq 0$, as we want.
This concludes the lemma.
\end{proof}

\subsection{Proof of Theorem \ref{thm:B=>C}} 
Let $X$ be a compact K\"ahler manifold which admits a real closed one-form $\alpha$ without zeros, i.e.\ condition (\ref{item:conj1:circle}) in Conjecture \ref{conj1} holds.
Since the pullback of $\alpha$ via a finite \'etale cover is again a real closed one-form without zeros, in order to prove (\ref{item:conj1:exact}), it suffices to show that $X$ carries a holomoprhic one-form $\omega$ such that $\wedge \omega$ is exact on cohomology.
For this, we may without loss of generality assume that $X$ is connected.

Since $\alpha$ has no zero on $X$, Theorem \ref{thm:Novikov-ineq} implies that there is a local system $L\in \Loc^0(X)$ that has no cohomology.
By the generic vanishing theorems \cite{GL,GL2,Ar,Si}, the locus of those local systems that have some cohomology are subtori, translated by torsion points, see \cite[Theorem 1.3]{Wa}.
It follows that for general $\omega \in H^0(X,\Omega_X^1)$, the local system $L(\omega)$ has no cohomology.
It thus follows from Lemma \ref{lem:coho-Lomega} that 
$$
H^{i-1}(X,\C)\stackrel{\wedge \omega}\longrightarrow H^{i}(X,\C)\stackrel{\wedge \omega}\longrightarrow H^{i+1}(X,\C) 
$$
is exact for all $i$, as we want.
This finishes the proof of Theorem \ref{thm:B=>C}.

\begin{remark}
Botong Wang points out that one can bypass the use of Theorem \ref{thm:Novikov-ineq} in the above argument by showing directly that if $X$ is a $C^\infty$-fibre bundle over the circle, then the pullback of a general local system on the circle has no cohomology on $X$.
\end{remark}

\begin{remark} \label{rem:generic-omega}
Let $X$ be a compact connected K\"ahler manifold.
As we have used above, the results in \cite{GL} imply that $(H^\ast(X,\C),\wedge \omega)$  is exact if and only if $L(\omega)$ has no cohomology.
The locus of such local systems is well understood by generic vanishing theory.
In particular, \cite[Theorem 1.3]{Wa} implies that the locus of those  holomorphic one-forms $\omega\in H^0(X,\Omega_X^1)$  for which $(H^\ast(X,\C),\wedge \omega)$ is not exact is a finite union of linear subspaces of the form $f_i^\ast H^0(T_i,\Omega_{T_i}^1)$, where $f_i:X\to T_i$ is a finite collection of holomorphic maps to complex tori $T_i$.
As a special case we see that if there is one holomorphic one-form $\omega$ on $X$ which makes $(H^\ast(X,\C),\wedge \omega)$ exact, then this holds for all forms in a non-empty Zariski open subset of $H^0(X,\Omega_X^1)$.
\end{remark}

\section{The case of surfaces}

\begin{proof}[Proof of Theorem \ref{thm:surfaces}] 
Let $X$ be a compact K\"ahler surface with a one-form $\omega \in H^0(X,\Omega_X^1)$ such that for any finite \'etale cover $\tau:X'\to X$,
\begin{align} \label{eq:ses:surfaces}
H^{i-1}(X',\C)\stackrel{\wedge \omega'}\longrightarrow H^{i}(X',\C)\stackrel{\wedge \omega'}\longrightarrow H^{i+1}(X',\C) 
\end{align}
is exact for all $i$, where $\omega':=\tau^\ast \omega$.
This implies $\chi(X,\Omega_X^p)=0$ for all $p$ and so $c_2(X)=0$.
%

Replacing $X$ by its connected components, we may without loss of generality assume that $X$ is connected.
The classification of surfaces (see \cite[Chapter VI.1]{barth-etal}) thus shows that only the following cases occur.

\textbf{Case 1.} $X$ is birational to a ruled surface over a curve $C$ of positive genus.

\textbf{Case 2.} $X$ is a minimal bi-elliptic surface or a complex $2$-torus.

\textbf{Case 3.} $X$ is a minimal properly elliptic surface.

In Case 1, exactness of (\ref{eq:ses:surfaces}) implies that $X$ is birational to a ruled surface over an elliptic curve $C$.
This implies $b_1(X)=2$.
Since $e(X)=0$, we conclude $b_2(X)=2$ and so $X$ is a minimal ruled surface over an elliptic curve.
In particular, since $\omega$ is nonzero, it must be a holomorphic one-form without zeros.

In Case 2, any nontrivial holomorphic one-form on $X$ has no zeros and so we are done because exactness of (\ref{eq:ses:surfaces}) implies $\omega\neq 0$, as before.  

In Case 3, the condition $c_2(X)=0$ implies by \cite[Proposition III.11.4]{barth-etal} that $X$ admits a fibration $\pi:X\to C$ to a curve $C$ such that the reduction of any fibre of $\pi$ is isomorphic to a smooth elliptic curve, but where multiple fibres are allowed.
Let $F$ be a general fibre of $\pi:X\to C$.
Suppose for the moment that the one-form $\omega$ restricts to a nonzero form on $F$.
In particular, the Albanese map $a :X\to \Alb(X)$ does not contract $F$ and the reduction of any fibre of $a $  is isomorphic to $F$.
Moreover, the restriction of $\omega$ to $F$ does not depend on the fibre and so it is nonzero everywhere. 
That is, $\omega$ has no zeros. 

It remains to deal with the case where $\omega$ restricts to zero on the fibres of $\pi:X\to C $.
In this case, $\omega=\pi^\ast \alpha$ for a one-form $\alpha$ on $C$.
Since cup product with $\omega$ is exact, $C$ must be an elliptic curve.
If $\pi$ is smooth, then $\omega$ has no zeros.
Otherwise, $\omega$ vanishes along the multiple fibres of $\pi$.
We may thus assume that $\pi$ has at least one multiple fibre. 

The multiple fibres of $\pi$ give rise to a orbifold structure on $C$.
Since $C$ is an elliptic curve, this orbifold is good and so there is a finite orbifold covering $C'\to C$ such that the orbifold structure on $C'$ is trivial, see e.g.\ \cite[Corollary 2.29]{CHK00}.
Let $X'$ be the normalization of the base change $X\times_CC'$.
Then, $X'$ is a smooth surface, $X'\to X$ is \'etale and $X'\to C'$ is an elliptic surface without singular fibres, see e.g.\ \cite[Proposition III.9.1]{barth-etal}. 
Since $\tau:X'\to X$ is finite \'etale, $(H^\ast(X',\C),\wedge \omega')$ is exact for $\omega':=\tau^\ast \omega$ by assumptions.  
On the other hand, since $\pi$ has singular fibres, $C'\to C$ is a branched covering with nontrivial branch locus and so $C'$ is a curve of genus $\geq 2$.
This is a contradiction, because $\omega'$ is a pullback of a one-form from $C'$. 
This finishes the proof of Theorem \ref{thm:surfaces}.
\end{proof}

\begin{corollary} \label{cor:classification:surfaces}
Let $X$ be a compact connected K\"ahler surface with a holomorphic one-form $\omega$ such that $(H^\ast(X,\C),\wedge \omega)$ is exact.
Then $\omega$ has no zeros and $(X,\omega)$ is given by one of the following:
\begin{enumerate}[(a)]
\item $X$ is a minimal ruled surface over an elliptic curve; \label{item:ruled}
\item $X$ is a complex $2$-torus; \label{item:torus} 
\item $X$ is a minimal elliptic surface $f:X\to C$ such that one of the following holds: \label{item:kod=0,1}
\begin{enumerate}[(i)]
\item $f$ is smooth, $C$ is an elliptic curve and $\omega\in f^\ast H^0(C,\Omega_C^1)$;\label{item:f-smooth}
\item $f$ is quasi-smooth, i.e.\ all singular fibres are multiple fibres, and the restriction of $\omega$ to the reduction of any fibre of $f$ is nonzero. \label{item:f-quasi-smooth}
\end{enumerate}
\end{enumerate} 
\end{corollary}
\begin{proof}
The classification into types (\ref{item:ruled}), (\ref{item:torus}) and (\ref{item:kod=0,1}) follows directly from the proof of Theorem \ref{thm:surfaces}, where we note that bi-elliptic surfaces fall in the class (\ref{item:f-smooth}).
The fact that $\omega$ has no zeros follows from this classification.
\end{proof}

\begin{corollary} 
In the notation of Corollary \ref{cor:classification:surfaces},
assume that $X$ is projective.
Then,
\begin{enumerate}[(a)] 
\setcounter{enumi}{3}
\item   $X$ admits a smooth morphism to a positive-dimensional abelian variety; \label{item:smoothmap}
\item if $\kappa(X)\geq 0$, then there is a finite \'etale cover $\tau:X'\to X$ which splits into a product $X'=A'\times S'$, where $A'$ is a positive-dimensional abelian variety and $S'$ is smooth projective.
\label{item:split}
\end{enumerate}
\end{corollary}
\begin{proof}
Note that item (\ref{item:smoothmap}) is clear in cases (\ref{item:ruled}), (\ref{item:torus}) and (\ref{item:f-smooth}) of Corollary \ref{cor:classification:surfaces}.
It remains to deal with case (\ref{item:f-quasi-smooth}).
In this case, since $X$ and hence $\Alb(X)$ are projective, $\Alb(X)$ is isogeneous to $E\times \operatorname{Jac}(C)$, where $E$ is an elliptic curve which is isogeneous to the reduction of any fibre of $f$.
It follows that there is a morphism $g:X\to E$ which restricts to an isogeny on the reduction of each fibre of $f:X\to C$.
Since $\omega$ restricts non-trivially to the reduction of any fibre of $f$, the morphism $g:X\to E$ must be smooth, as we want. 

It clearly suffices to prove item (\ref{item:split}) in the case (\ref{item:kod=0,1}) of Corollary \ref{cor:classification:surfaces}.
In this case, there is a finite \'etale cover $X'\to X$, such that $\Alb(X')\cong E\times \operatorname{Jac}(C')$ for a smooth projective curve $C'$ which maps finitely to $C$.
Moreover, the Albanese map identifies $X'$ to the product $E\times C'$, as we want.
This concludes the corollary.
\end{proof}


\section{Proof of Theorem \texorpdfstring{\ref{thm:C=>...}}{1.4}}

\subsection{Preliminaries}
We will use the following lemma.

\begin{lemma} \label{lem:support}
Let $K^\ast$ be a bounded complex of sheaves on a manifold $X$.
Let $Z,Z'\subset X$ be closed subsets with $Z\cap Z'=\emptyset$, such that
$$
\supp \mathcal H^i(K^*)\subset Z\cup Z'
$$ 
for all $i$.
Then the differentials $d_r:E^{p,q}_{r}\to E_r^{p+r,q-r+1}$ in the spectral sequence
$$
E_2^{p,q}=H^p(X,\mathcal H^q(K^*))\Rightarrow \mathbb H^{p+q}(X,K^\ast) 
$$
respect the natural decompositions
$$
E_2^{p,q}=H^p(Z,\mathcal H^q(K^*)|_Z)\oplus H^p(Z',\mathcal H^q(K^*)|_{Z'}) .
$$ 
\end{lemma}

\begin{proof}
%
Let $i:Z\to X$ and $j:Z'\to X$ be the inclusions.
Then the natural map of complexes
$$
K^*\twoheadrightarrow i_\ast i^{-1}K^\ast \oplus j_\ast j^{-1}K^\ast 
$$
is a quasi-isomorphism.
This proves the lemma, because the spectral sequence depends only on the class of $K^*$ in the derived category of sheaves on $X$.
\end{proof}

\subsection{Item (\ref{item:thm:Z(omega)}) of Theorem \ref{thm:C=>...}}

Let $X$ be a compact connected K\"ahler manifold and let $\omega$ be a holomorphic one-form on $X$ with associated local system $L(\omega)$.
Recall the isomorphism 
$$
H^k(X,L(\omega))\cong \mathbb H^k(X,(\Omega_X^\ast,\omega\wedge-)) .
$$ 
The above hypercohomology is computed by a spectral sequence with $E_2$-page
\begin{align} \label{def:ss:H^qK}
E_2^{p,q}:=H^p(X, \mathcal H^q(K^*))\Rightarrow H^{p+q}(X,L(\omega)) ,
\end{align}
where $K^*:= (\Omega_X^\ast,\omega\wedge-)$ and $\mathcal H^q(K^*)$ denotes the $q$-th cohomology sheaf of that complex.
In particular, $\mathcal H^q(K^*)=0$ if $\omega \wedge-$ is exact on holomorphic $q$-forms and the latter holds if $\omega$ has no zeros.
More precisely, this shows that $\mathcal H^q(K^*)$ are sheaves that are supported on the zero locus $Z(\omega)$ of $\omega$. 

\begin{lemma} \label{lem:H^n(K)}
We have $\mathcal H^n(K^\ast)\cong \Omega^n_X|_Z$.
\end{lemma}
\begin{proof}
Locally $\omega=\sum_{i=1}^n f_i dx_i$.
We are interested in the cokernel of
$$
\Omega_X^{n-1}\to \Omega_X^n,\ \ \alpha \mapsto \sum_{i=1}^n f_i dx_i\wedge \alpha .
$$
The image of the above map is clearly spanned by $f_i dx_1\wedge \dots \wedge dx_n$ with $i=1,\dots ,n$.
Hence, $\mathcal H^n(K^\ast)$ is the quotient of $\Omega_X^n$ by the subsheaf $I_Z\otimes_{\mathcal O_X}  \Omega_X^n$, where $I_Z$ denotes the ideal sheaf of $Z$.
Hence,
$$
\mathcal H^n(K^\ast)\cong \Omega_X^n\otimes_{\mathcal O_X} \mathcal O_Z = \Omega_X^n|_Z .
$$
This proves the lemma.
\end{proof}

\begin{proof}[Proof of item (\ref{item:thm:Z(omega)}) in Theorem \ref{thm:C=>...}] 
Let $Z\subset Z(\omega)$ be a connected component of the zero locus of $\omega$.
Then we can write $Z(\omega)=Z\cup Z'$, where $Z$ and $Z'$ are disjoint closed subsets of $X$.

Consider the spectral sequence (\ref{def:ss:H^qK}).
By Lemma \ref{lem:H^n(K)}, we have
$$
H^d(X,\Omega_X^n|_Z)\hookrightarrow E^{d,n}_2.
$$
Using Lemma \ref{lem:support}, one easily checks that this term survives on the infinity page and we get
$$
H^d(X,\Omega_X^n|_Z)\hookrightarrow E^{d,n}_\infty.
$$
By Lemma \ref{lem:coho-Lomega}, exactness of $(H^\ast(X,\C),\wedge \omega)$ implies $H^{i}(X,L(\omega))=0$ for all $i$.
Hence, $E^{d,n}_\infty=0$, and so $H^d(Z,\Omega_X^n|_Z)=0$, as we want.
\end{proof}

\begin{corollary}\label{cor:Z^red}
Let $X$ be a compact K\"ahler manifold and let $\omega\in H^0(X,\Omega_X^1)$ 
such that the complex $(H^\ast(X,\C),\wedge \omega)$ given by cup product with $ \omega$ is exact.
Let $Z\subset Z(\omega)$ be a connected component of the zero locus of $\omega$, and let $d=\dim Z$.
Then
$$
H^d(Z',\omega_X|_{Z'})=0 ,
$$
for any irreducible component $Z'$ of the reduced scheme $Z^{\red}$.
\end{corollary}
\begin{proof}
 Consider the long exact sequence, associated to the short exact sequence
$$
0\longrightarrow \omega_X|_Z\otimes \mathcal I_{Z'} \longrightarrow \omega_X|_Z \longrightarrow \omega_{X}|_{Z'}\longrightarrow 0 .
$$
By item (\ref{item:thm:Z(omega)}), $H^d(Z,\omega_X|_{Z})=0$.
Moreover, $H^{d+1}(Z,\omega_X|_Z\otimes \mathcal I_{Z'})=0$ because of dimension reasons.
This implies $H^d(Z',\omega_X|_{Z'})=0$, as we want.
\end{proof}


\begin{corollary}\label{cor:a(X)}
Let $X$ be a compact K\"ahler manifold with a holomorphic map $f:X\to A$ to a complex torus $A$.
Let $\omega\in H^0(A,\Omega_A^1)$ 
such that the complex $(H^\ast(X,\C),\wedge f^\ast \omega)$ given by cup product with $f^\ast \omega$ is exact.

Then the restriction of $\omega$ to $f(X)\subset \Alb(X)$ does not vanish at a point $y\in f(X)$ such that the fibre $F:=f^{-1}(y)$ is smooth with trivial normal bundle (the locus of such points $y\in f(X)$ is Zariski dense in $f(X)$).
\end{corollary}
\begin{proof}
Assume that $\omega$ vanishes at a point $y\in f(X)$ such that the fibre $F:=f^{-1}(y)$ is smooth with trivial normal bundle.
Then $F\subset Z(f^\ast \omega)^{\red}$ is a connected component.
This contradicts Corollary \ref{cor:Z^red}, because
$$
H^{\dim F}(F,\omega_X|_F)=H^{\dim F}(F,\omega_F) 
\neq 0,
$$
by Serre duality, where we used that $F$  has trivial normal bundle.
\end{proof}

\subsection{Item (\ref{item:thm:fibred}) of Theorem \ref{thm:C=>...}}

Let $f:X\to A$ be a holomorphic map to a complex torus $A$ and assume that there is a one-form $\omega \in f^\ast H^0(A,\Omega_A^1)$ such that $(H^\ast(A,\C),\wedge \omega)$ is exact.
Since exactness is an open property, $(H^\ast(A,\C),\wedge \omega')$ is exact for any general $\omega' \in f^\ast H^0(A,\Omega_A^1)$.

Let  $Y:=f(X)$ and fix a general point $y\in Y$.
There are countably many non-trivial linear subspaces
$$
\{0\}\neq W_i\subset T_{A,y}
$$
such that there is a morphism of complex tori $\pi_i:A\to B_i$ with $\ker((d \pi_i)_y)=W_i$.

For a contradiction, we assume that $Y$ is not fibred by tori.
This implies that the tangent space $T_{Y,y} $ does not contain any of the $W_i$.
We may thus choose a one-form $\omega'\in H^0(A,\Omega_A^1)$, such that $\omega'$ vanishes on $T_{Y,y}\subset T_{A,y}$, but which is non-trivial on each $W_i$.
Let $Z\subset Z(\omega')$ be an irreducible component which contains $y$.
Then $\omega'$ vanishes on $Z$ and hence on the subtorus $\langle Z\rangle\subset A$, generated by $Z$. 
If $Z$ was positive-dimensional, then $ T_{\langle Z\rangle,y}=W_i$ for some $i$, which contradicts the fact that $\omega'$ does not vanish on $W_i$.
Hence, $Z$ is zero-dimensional and so $y$ is an isolated zero of $\omega'|_Y$.
But this implies that a small perturbation of $\omega'|_Y$ has an isolated zero in some neighbourhood of $y$.
Hence, a general one-form $\omega \in H^0(A,\Omega_A^1)$ has the property that $Z(\omega|_Y)$ contains a general point of $Y$ as a connected component.
This contradicts Corollary \ref{cor:a(X)}, which finishes the proof.


\appendix

\section*{Appendix, written jointly with Hsueh-Yung Lin}

\setcounter{section}{1}

\setcounter{theorem}{0}

In this appendix we prove the following.

\begin{theorem}\label{thm:appendix}
Let $X$ be a compact connected K\"ahler manifold.
Assume that $\omega\in H^0(X,\Omega_X^1)$ satisfies condition (\ref{item:conj1:exact}). 
Then $ \dim Z(\omega) \le \dim X - 2$.  
\end{theorem}

By Theorem \ref{thm:C=>...}, we also have $1\leq \dim Z(\omega)$. 
If $\dim X=2$, the above theorem thus implies $Z(\omega)=\emptyset$, which yields a new proof of Theorem \ref{thm:surfaces}, without using the Enriques-Kodaira classification. 

We start with the following auxiliary result; 
the same argument appeared in the last two paragraphs in the proof of Theorem \ref{thm:surfaces}, as well as in \cite[Proposition 6.4]{HS}.

\begin{lemma}\label{lem-int}
Let $X$ be a compact connected K\"ahler manifold with a morphism $f:X\to E$ to an elliptic curve $E$ with irreducible fibres.
Assume that there is a one-form $\alpha \in H^0(E,\Omega_E^1)$ such that $\omega:=f^\ast \alpha$ satisfies condition (\ref{item:conj1:exact}). 
Then $f$ has reduced fibres.
\end{lemma}

\begin{proof}
Let $\Delta$ be the set of points $t \in E$ such that $f^{-1}(t)$ is a multiple fibre and let $m_t$ be its multiplicity. 
This gives rise to an orbifold structure on $E$.
Since $E$ is an elliptic curve, this orbifold sutrcture is good (see e.g.\ \cite[Corollary 2.29]{CHK00}) and so there is a finite cover $C\to E$ which locally above each point of $t\in \Delta$ is ramified of order $m_t$.
A local computation shows that the normalization $\tilde X$ of $X\times_EC$ is \'etale over $X$, cf.\ \cite[Proposition III.9.1]{barth-etal}.
There is a natural map $\tilde f:\tilde X\to C$ and our assumptions imply that 
there is a one-form $\omega\in H^0(C,\Omega_C^1)$ such that $(H^\ast(\tilde X,\C),\wedge \tilde f^\ast \omega)$ is exact.
This implies $g(C)=1$ and so $\Delta = \emptyset$, as we want.  
\end{proof}

\begin{proof}[Proof of Theorem~\ref{thm:appendix}]
Assume for the contrary that there is a prime divisor $D\subset Z(\omega)$.
Let $f:X\to A$ be a morphism to a complex torus such that $\omega=f^\ast \alpha$ for some $\alpha\in H^0(A,\Omega_A^1)$, and assume that $\dim A$ is minimal with that property.

Since $\omega|_D=0$, we have $\alpha|_{\langle f(D)\rangle}=0$, where $\langle f(D)\rangle\subset A$ denotes the subtorus generated by $f(D)$.
Hence, $\omega$ is the pullback of a one-form from $A/\langle f(D)\rangle$.
Minimality of $\dim A$ thus shows that $f(D)$ is a point.
It then follows from \cite[Lemma 2.4]{HS} that $A$ is an elliptic curve.
Moreover, up to replacing $f$ by its Stein factorization, we may  by \cite[Corollary 2.5]{HS} assume that all fibres of $f$ are irreducible. 
Hence, $f$ has reduced fibres by Lemma~\ref{lem-int}.  
Since $A$ is an elliptic curve, $Z(\omega)$ is contained in the singular locus of $f$, which has codimension at least two, because the fibres of $f$ are reduced.
This is a contradiction, which concludes the theorem.
\end{proof}

\end{document}